\documentclass[11pt]{article}
\usepackage{amsthm, amsmath, amssymb, amsfonts, url, booktabs, tikz, setspace, fancyhdr, bm}
\usepackage{geometry}
\usepackage{hyperref, enumerate}
\usepackage[shortlabels]{enumitem}
\usepackage[babel]{microtype}
\usepackage[english]{babel}
\usepackage[capitalise]{cleveref}
\usepackage{comment}
\usepackage{bbm}
\usepackage{csquotes}

\usepackage{tikz}
\usepackage{graphicx}
\usepackage{float}
\usepackage[dvipsnames]{xcolor}
\usepackage{soul}
\usepackage{mathtools}
\usetikzlibrary{positioning, arrows.meta, shapes.geometric}
\usepackage[normalem]{ulem}

\counterwithin{figure}{section}

\newtheorem{theorem}{Theorem}[section]

\newtheorem{conj}[theorem]{Conjecture}
\newtheorem{lemma}[theorem]{Lemma}

\newtheorem{claim}[theorem]{Claim}

\newtheorem{remark}[theorem]{Remark}
\newtheorem{corollary}[theorem]{Corollary}
\newtheorem{proposition}[theorem]{Proposition}
\usetikzlibrary{decorations.pathmorphing}

\theoremstyle{definition}
\newtheorem{defn}[theorem]{Definition}
\newtheorem*{defn-non}{Definition}

\newcommand{\norm}[1]{\left\lVert #1\right\rVert}

\newlist{Case}{enumerate}{2}
\setlist[Case, 1]{%
    label           =   {\bfseries Case \arabic*.},
    labelindent=1em ,labelwidth=1.3cm, labelsep*=1em, leftmargin =!
}
\setlist[Case, 2]{%
    label           =   {\bfseries Subcase \arabic{Casei}.\arabic*.},
    labelindent=-1em ,labelwidth=1.3cm, labelsep*=1em, leftmargin =!
}

\newcommand{\E}{{\rm E}}

\newcommand{\ex}{\mathrm{ex}}

\newcommand{\cT}{\mathcal{T}}

\newcommand*{\abs}[1]{\lvert#1\rvert}

\title{The Tur\'an number of the Cartesian product of trees via star-flip}
\author{
Lanchao Wang\thanks{School of Mathematics, Nanjing University, Nanjing, China, and ECOPRO, Institute for Basic Science, 55 Expo-ro, Yuseong-gu, Daejeon, 34126, Korea. Email: lanchaowang@foxmail.com.}
\and
Caihong Yang\thanks{School of Mathematics and Physics, China University of Geosciences, Wuhan, China. Email: yangch@cug.edu.cn.}
}

\begin{document}
\date{}
\maketitle

\begin{abstract}
Motivated by Erd\H{o}s's conjecture  on the Tur\'an number
of
degenerate bipartite graphs, Brada\v{c}, Janzer, Sudakov and Tomon
proved that $
\ex(n,T \Box  P)=\Theta_{T,P}(n^{3/2})$ for every
nontrivial tree $T$ and every nontrivial path $P$, and conjectured that
the same order of magnitude holds for the Cartesian product of any two nontrivial
trees. We prove their conjecture. More generally, for every integer $r\ge2$, we introduce a class of
bipartite $r$-degenerate graphs, called $r$-star-flip graphs, that are obtained
from a seed tree by a sequence of local vertex-duplication
operations. We prove that every fixed $r$-star-flip graph $H$ satisfies
$\ex(n,H)=O_H(n^{2-1/r})$. Every Cartesian product of two trees is a $2$-star-flip graph, while the star-flip class also contains graphs that do not arise as such products. As a further application, our framework yields a new proof of
F\"uredi's theorem: if $H$ is a fixed bipartite graph in which at most
one vertex in one colour class has degree greater than $r$, then
$\ex(n,H)=O_H(n^{2-1/r})$. The key ingredient is a conditional-resampling procedure that extends
the tree branching random walk on the seed tree to a random
homomorphism of the entire star-flip graph, while preserving the
branching-random-walk distribution on every live tree.

\end{abstract}

\section{Introduction}
For a graph $F$ and a positive integer $n$, the \emph{Tur\'an number}
of $F$, denoted by $\mathrm{ex}(n,F)$, is the maximum number of edges
in an $F$-free graph on $n$ vertices. The study of $\mathrm{ex}(n,F)$
is a central topic in extremal combinatorics. The celebrated
Erd\H{o}s--Stone--Simonovits theorem~\cite{ES1966,ES1946} states that
$\mathrm{ex}(n,F)=\left(1-\frac{1}{\chi(F)-1}+o(1)\right)\binom{n}{2}$,
where $\chi(F)$ is the chromatic number of $F$. This theorem determines
$\mathrm{ex}(n,F)$ asymptotically whenever $\chi(F)\ge3$. However, for
bipartite $F$, it only gives $\ex(n,F)=o(n^2)$, and determining even
the correct order of magnitude is notoriously difficult. The
K\H{o}v\'ari--S\'os--Tur\'an theorem~\cite{KST54} implies that every
fixed bipartite graph has a genuinely subquadratic Tur\'an number, but
there is no general principle that predicts the correct exponent; see
the survey of F\"uredi and Simonovits~\cite{FurediSimonovits2013}. Determining the possible Tur\'an exponents of bipartite graphs is
itself a major problem; see, for example,
\cite{BukhConlon,KangKimLiu,JiangQiu}.

A graph $H$ is called \emph{$r$-degenerate} if each of its subgraphs
has minimum degree at most $r$. Generalising the
K\H{o}v\'ari--S\'os--Tur\'an theorem, Erd\H{o}s~\cite{Erdos1967}
proposed the following conjecture in 1967.
\begin{conj}[Erd\H{o}s \cite{Erdos1967}]\label{conj-erdos}
Let $H$ be a bipartite $r$-degenerate graph. Then
$\ex(n,H)=O_H(n^{2-1/r})$.
\end{conj}

The conjecture is known when one side of a bipartition of $H$ has
maximum degree at most $r$, by work of F\"uredi~\cite{Furedi1991} and
Alon, Krivelevich and Sudakov~\cite{AKS2003}. It is also known for
several further families, including the $r$-degenerate blow-ups of
trees studied by Grzesik, Janzer and Nagy~\cite{GJN2002} and the more
general tree-degenerate graphs introduced by Jiang and
Longbrake~\cite{JL2023}. Related variants have recently been studied
in \cite{GaoLiuMaPikhurko2025,HouHuLiLiuYangZhang2026}. Nevertheless,
the conjecture remains open in general even for $r=2$. In this paper, we
introduce a new family of bipartite $r$-degenerate graphs for which the
conjectured bound holds.

Cartesian products provide a natural source of motivating examples.
For graphs $F$ and $G$, the Cartesian product $F\Box  G$ is the graph with vertex
set $V(F)\times V(G)$, in which $(u,v)$ and $(u',v')$ are adjacent  if and
only if either $u=u'$ and $vv'\in E(G)$, or $v=v'$ and $uu'\in E(F)$.
Brada\v{c}, Janzer, Sudakov and Tomon \cite{BJST2023} proved that
$\ex(n,T\Box  P)=\Theta_{T,P}(n^{3/2})$ for every nontrivial tree $T$
and every nontrivial path $P$. Their proof introduced a novel use of the tensor-power trick and exploited the linear structure of the path $P$ through a ladder construction. They further conjectured the following extension.

\begin{conj}[Brada\v{c}, Janzer, Sudakov and Tomon
\cite{BJST2023}]\label{conj:tree-product}
For any two trees $T$ and $S$ (each with at least one edge), there exist positive real
numbers $c$ and $C$ such that
$cn^{3/2}\leq \ex(n,T\Box S)\leq Cn^{3/2}$.
\end{conj}

The lower bound follows from the fact that $T\Box S$ contains a copy
of $C_4$, and hence every $C_4$-free graph is also $T\Box S$-free. In the special case of the grid $P_t\Box  P_t$, their proof gave the
upper bound $\ex(n,P_t\Box  P_t)\le e^{O(t^5)}n^{3/2}$, and they asked
to determine the correct dependence on $t$. Gao, Janzer, Liu and Xu
\cite{GJLX2025} improved this bound to
$\ex(n,P_t\Box  P_t)\le O(t^{3/2})n^{3/2}$ and also developed related
embedding methods for several graphs arising from geometric
quadrangulations, such as the even prisms $C_{2\ell}\Box  K_2$. More
recently, Zhao, Cheng, Chi, Gy\H{o}ri, Tompkins and Wang
\cite{ZCCGTW2026} studied the graph $K_2\Box S_t$, which consists of $t$ copies of
$C_4$ sharing a common edge, and obtained results on both ordinary and bipartite Tur\'{a}n numbers.

We prove Conjecture~\ref{conj:tree-product} as an application of a more
general embedding theorem for a new class of bipartite $r$-degenerate
graphs, thereby providing further evidence for Conjecture \ref{conj-erdos}.
To motivate the construction, consider the grid $P_t\Box P_t$. Start
with the tree formed by its first row and first column. Each remaining
vertex $v_{i,j}$ can then be introduced by duplicating
$v_{i-1,j-1}$ with respect to the two-edge star
$v_{i-1,j}v_{i-1,j-1}v_{i,j-1}$. Iterating this operation produces
the entire grid. The same operation may be performed inside any
currently available tree, leading to the following class.

\begin{defn}[Star-flip presentations and star-flip graphs]
\label{def:star-flip}
Fix an integer $r\ge2$. An \emph{$r$-star-flip presentation} of a
graph $H$ consists of a sequence
$
H_0\subset H_1\subset\cdots\subset H_m=H
$
together with families $\cT_0,\cT_1,\ldots,\cT_m$ of trees, whose
members are called \emph{live trees}. These objects are 
constructed
recursively as follows.

The initial graph $H_0=R$ is a tree, called the \emph{seed tree}, and
$\cT_0$ is the family of all subtrees of $R$.
Suppose that $H_{i-1}$ and $\cT_{i-1}$ have been defined for some
$i\in[m]$. Choose a live tree $T_i\in\cT_{i-1}$ and a vertex
$c_i\in V(T_i)$ such that
$
1\le k_i:=d_{T_i}(c_i)\le r.
$
Introduce a new vertex $x_i\notin V(H_{i-1})$ and join it precisely to
the vertices of $N_{T_i}(c_i)$. Thus
\[
V(H_i)=V(H_{i-1})\cup\{x_i\}
\ \ \text{ and }\ \ 
E(H_i)=E(H_{i-1})
\cup\{x_i a:a\in N_{T_i}(c_i)\}.
\]
We now update the family of live trees. Let
$
\mathcal S_i
=
\{T\in\cT_{i-1}:c_i\in V(T)\text{ and }N_T(c_i)=N_{T_i}(c_i)\}.
$
For every $T\in\mathcal S_i$, let $T^{x_i}$ be the tree obtained from $T$ by
replacing $c_i$ with $x_i$. Every old live tree remains live, and we additionally declare every
subtree of each $T^{x_i}$ with $T\in\mathcal S_i$ to be live. Hence
\[
\cT_i
=
\cT_{i-1}
\cup
\bigcup_{T\in\mathcal S_i}
\{Q:Q\text{ is a subtree of }T^{x_i}\}.
\]
The passage from $(H_{i-1},\cT_{i-1})$ to
$(H_i,\cT_i)$ is called a \emph{star flip of size $k_i$}, performed
at $c_i$ with neighbour set $N_{T_i}(c_i)$. A graph $H$ is called an \emph{$r$-star-flip graph} if it admits an
$r$-star-flip presentation.
\end{defn}

It follows directly from the construction that every $r$-star-flip
graph is bipartite and $r$-degenerate. Indeed, fix a bipartition of the
seed tree $R$ and extend it inductively. At step $i$, the vertices
$a_{i,1},\ldots,a_{i,k_i}$ all lie in the colour class opposite to
$c_i$ in the live tree $T_i$, so placing $x_i$ in the same colour class
as $c_i$ preserves bipartiteness. For degeneracy, delete the added
vertices in the reverse order $x_m,x_{m-1},\ldots,x_1$. When $x_i$ is deleted, every neighbour of $x_i$ that was introduced at a step $j>i$ has already been removed, while $a_{i,1},\ldots,a_{i,k_i}$ are still present.
Hence its remaining degree is exactly $k_i\le r$. The remaining graph
is the seed tree $R$, which can be deleted leaf by leaf. Thus $H$ is
$r$-degenerate.

Our main result establishes Erd\H{o}s's conjecture for all
$r$-star-flip graphs.

\begin{theorem}\label{thm:main}
Fix an integer $r\ge2$. If $H$ is an $r$-star-flip graph, then
$\ex(n,H)=O_H(n^{2-1/r})$.
\end{theorem}

The Cartesian product of two trees can be constructed using only
$2$-star flips. This proves Conjecture~\ref{conj:tree-product} and, in fact, yields an
upper bound whose leading constant is
$O\bigl(|V(T)|\,|V(S)|\bigr)$.

\begin{theorem}\label{cor:tree-products}
Let $T$ and $S$ be two trees, each containing at least one edge. Then
\[
\ex(n,T\Box  S)=\Theta_{T,S}(n^{3/2}).
\]
\end{theorem}
To illustrate that the class of star-flip graphs extends well beyond
Cartesian products of trees, we recover a classical theorem of
F\"uredi~\cite{Furedi1991}. This theorem was reproved using dependent
random choice by Alon, Krivelevich and Sudakov~\cite{AKS2003}, and also
follows from the more general random-walk framework of Grzesik, Janzer
and Nagy~\cite{GJN2002}. Although our approach is related in spirit to
the latter, it uses a different mechanism: instead of a random walk on
an auxiliary graph of $r$-sets obtained through supersaturation, we
work directly in the host graph and extend a tree branching random
walk through conditional resampling.

\begin{corollary}[F\"uredi \cite{Furedi1991}]
\label{cor:exceptional-degree}
Fix an integer $r\ge2$, and let $L$ be a fixed bipartite graph. Suppose that one colour class of
$L$ contains a vertex $v_0$ such that every other vertex in that colour
class has degree at most $r$. Then
$\ex(n,L)=O_L(n^{2-1/r})$.
\end{corollary}

To deduce this from \Cref{thm:main}, let $A$ be the colour class
opposite to $v_0$, and add all missing edges between $v_0$ and $A$.
Take the resulting star with centre $v_0$ and leaf set $A$ as the seed
tree. For every other vertex $v$ in the colour class of $v_0$, the
substar with centre $v_0$ and leaf set $N_L(v)$ is live. Duplicating
its centre introduces a vertex whose neighbourhood is precisely
$N_L(v)$, and this is a star flip of size $d_L(v)\le r$. Repeating this
for all such vertices produces an $r$-star-flip supergraph of $L$.
Hence \Cref{thm:main} implies
$\ex(n,L)=O_L(n^{2-1/r})$.

The proof of \Cref{thm:main} constructs a random homomorphism from $H$
to the host graph $G$ and shows that it is injective with positive
probability. We begin with the tree branching random walk, a probability
measure on tree homomorphisms used in work on Sidorenko's conjecture
\cite{CKLL2018}, and extend it from the seed tree to the entire
star-flip graph. At each flip, the new vertex is sampled from the
conditional distribution of the vertex that it duplicates. This
preserves the branching-random-walk distribution on every live tree.

The rest of the paper is organised as follows. In
Section~\ref{sec:definition}, we develop the random star-flip embedding
and establish its properties. In Section~\ref{sec:collisions},
we prove Theorem~\ref{thm:main}. In Section~\ref{sec:applications}, we
derive the applications to Cartesian products of trees and to bipartite
graphs with one exceptional vertex.

\paragraph{Additional note.}
Shortly after the first version of this paper was posted on arXiv,
OpenAI~\cite{OpenAI2026} announced a counterexample to Erd\H{o}s's
Conjecture~\ref{conj-erdos}.  In light of this development, the results of the present
paper should be viewed not as evidence for the conjecture in full
generality, but as part of the broader problem of identifying natural
classes of bipartite $r$-degenerate graphs for which the bound
$O_H(n^{2-1/r})$ continues to hold.
\section{The star-flip random embedding}\label{sec:definition}
In this section, we develop a random embedding method for
$r$-star-flip graphs. We first recall the tree branching random walk.
We then identify the conditional distribution of the image of a vertex
given the images of all other vertices, and use this distribution to
expose the vertices introduced by the flips. Throughout this section, let $G$ be a graph with no isolated vertices,
and write $N=|V(G)|$. We write $d_G(x)$, or simply $d(x)$, for the
degree of $x$ in $G$, and $\delta(G)$ for the minimum degree of $G$.

\subsection{Tree branching random walks}\label{sec:brw}
We use the tree branching random walk described by Conlon, Kim, Lee and
Lee in \cite[Section~2.1]{CKLL2018}.

\begin{defn}\label{def:brw}
Let $T$ be a tree rooted at $\rho\in V(T)$. A random homomorphism
$\Phi:T\to G$ is generated as follows. First choose the
image of $\rho$
according to $\Pr(\Phi(\rho)=x)=d(x)/2e(G)$ for every $x\in V(G)$.
Orient every edge of $T$ away from $\rho$. Then, conditional on the
image of its parent, map each child independently and uniformly to a
neighbour of that image. We denote the resulting probability
distribution on $\operatorname{Hom}(T,G)$ by $\mu_T$.
\end{defn}

The degree-proportional choice of the image of the root ensures that
the resulting distribution is independent of the choice of root. For
completeness, we include a short proof.

\begin{lemma}[\cite{CKLL2018}]\label{lem:brw-formula}
For every homomorphism $\phi:T\to G$,
\begin{equation}\label{eq:brw-formula}
\mu_T(\phi)=\frac{1}{2e(G)}\prod_{v\in V(T)}\frac{1}{d(\phi(v))^{d_T(v)-1}}.
\end{equation}
Thus $\mu_T$ is independent of the choice of root.
\end{lemma}

\begin{proof}
The choice of the root image contributes $d(\phi(\rho))/(2e(G))$. Every
oriented edge from a parent $v$ to a child contributes the factor
$1/d(\phi(v))$. The root has $d_T(\rho)$ children, whereas every
non-root vertex $v$ has $d_T(v)-1$ children. Multiplying all factors
gives \eqref{eq:brw-formula}.
\end{proof}
The distribution is consistent under restriction to subtrees.
\begin{lemma}[\cite{CKLL2018}]\label{lem:subtree}
Let $T'$ be a subtree of a tree $T$, and let $\Phi$ be sampled
according to $\mu_T$. Then $\Phi|_{T'}$ is distributed according to
$\mu_{T'}$.
\end{lemma}
The next lemma bounds the probability that two vertices of the tree receive the same image.
\begin{lemma}\label{lem:seed-collision}
Let $\Phi:T\to G$ be sampled according to $\mu_T$. Then, for any
distinct vertices $u,v\in V(T)$,
$\Pr(\Phi(u)=\Phi(v))\le1/\delta(G)$.
\end{lemma}

\begin{proof}
Let $u=w_0,w_1,\ldots,w_q=v$ be the unique path from $u$ to $v$ in
$T$. By Lemma~\ref{lem:subtree}, the restriction of $\Phi$ to this path is
distributed as a tree branching random walk. By root independence, we
may root the path at $u$ and write $X_j=\Phi(w_j)$.

If $q=1$, then $X_0\ne X_1$. Suppose that $q\ge2$. Conditional on
$X_0,\ldots,X_{q-1}$, the vertex $X_q$ is chosen uniformly from
$N_G(X_{q-1})$. Hence
\[
\Pr(X_q=X_0\mid X_0,\ldots,X_{q-1})
=\frac{\mathbf 1_{\{X_0\in N_G(X_{q-1})\}}}{d(X_{q-1})}
\le\frac1{\delta(G)}.
\]
Taking expectations over $(X_0,\ldots,X_{q-1})$ proves the result.
\end{proof}

\subsection{Conditional resampling on trees}
Let $T$ be a tree, let $c\in V(T)$, and write
$N_T(c)=\{a_1,\ldots,a_k\}$, where $k=d_T(c)$. For an ordered
$k$-tuple $\bm u=(u_1,\ldots,u_k)\in V(G)^k$ with $\cap^k_{j=1}N_G(u_j)\neq \emptyset $,
define
\begin{equation}\label{eq:Zk}
Z_k(\bm u)=\sum_{z\in\bigcap_{j=1}^kN_G(u_j)}\frac{1}{d(z)^{k-1}}
\end{equation}
and
\begin{equation}\label{eq:qk}
q_{\bm u}(z)=
\frac{\mathbf 1_{\{z\in\bigcap_{j=1}^kN_G(u_j)\}}}
{Z_k(\bm u)}\cdot\frac{1}{d(z)^{k-1}}.
\end{equation}

\begin{lemma}\label{lem:conditional-law}
With the above notation, let $\Phi:T\to G$ be distributed according to $\mu_T$. Conditional on
the random vector $\Phi|_{V(T)\setminus\{c\}}$, the distribution of
$\Phi(c)$ is
$q_{(\Phi(a_1),\ldots,\Phi(a_k))}$. More precisely, for every
$z\in V(G)$,
\[
\Pr\bigl(\Phi(c)=z\mid \Phi|_{V(T)\setminus\{c\}}\bigr)
=
q_{(\Phi(a_1),\ldots,\Phi(a_k))}(z).
\]
Consequently, replacing $\Phi(c)$ by a conditionally independent sample
from this distribution leaves the distribution of the entire random
homomorphism unchanged; in particular, the resulting homomorphism is
still distributed according to $\mu_T$.
\end{lemma}

\begin{proof}
Fix an admissible map
$\psi:V(T)\setminus\{c\}\to V(G)$, that is,
$
\Pr\bigl(\Phi|_{V(T)\setminus\{c\}}=\psi\bigr)>0.
$
Set $u_j=\psi(a_j)$ for $j\in[k]$. Conditional on
$\Phi|_{V(T)\setminus\{c\}}=\psi$, all factors in
\eqref{eq:brw-formula} are fixed except those involving the value
$z=\Phi(c)$. These factors are
\[
\frac{1}{d(z)^{k-1}}\prod_{j=1}^k\mathbf 1_{\{zu_j\in E(G)\}}.
\]
Hence, for every $z\in V(G)$,
\begin{align*}
\Pr\bigl(\Phi(c)=z\mid
\Phi|_{V(T)\setminus\{c\}}=\psi\bigr)
&=
\frac{
\Pr\bigl(\Phi(c)=z,\,
\Phi|_{V(T)\setminus\{c\}}=\psi\bigr)
}{
\displaystyle
\sum_{w\in V(G)}
\Pr\bigl(\Phi(c)=w,\,
\Phi|_{V(T)\setminus\{c\}}=\psi\bigr)
}\\
&=
\frac{\frac{1}{
d(z)^{k-1}}\prod_{j=1}^k\mathbf 1_{\{zu_j\in E(G)\}}
}{
\displaystyle
\sum_{w\in V(G)}
\frac{1}{d(w)^{k-1}}\prod_{j=1}^k\mathbf 1_{\{wu_j\in E(G)\}}
}\\
&=
q_{(u_1,\ldots,u_k)}(z).
\end{align*}

Now replace $\Phi(c)$, conditional on the images of the remaining
vertices, by an independent sample from this same conditional
distribution. Since the distribution of the remaining images is
unchanged and the conditional distribution of $\Phi(c)$ given them is
also unchanged, the resulting random homomorphism is still distributed
according to $\mu_T$.
\end{proof}
We shall use the following averaged bound on the largest value of the
conditional distribution. For a probability distribution $q$ on
$V(G)$, write
$
\norm{q}_\infty:=\max_{z\in V(G)}q(z).
$

\begin{lemma}\label{lem:max-atom}
Let $S_k$ be the star with centre $c$ and leaves
$a_1,\ldots,a_k$. Let $\Phi:S_k\to G$ be distributed according to
$\mu_{S_k}$, and put
$
U=(\Phi(a_1),\ldots,\Phi(a_k)).
$
Then
\[
\E\bigl[\norm{q_U}_\infty\bigr]
\le
\frac{N^{k-1}}{\delta(G)^k}.
\]
\end{lemma}

\begin{proof}
Write $Z=\Phi(c)$. By \eqref{eq:brw-formula}, for every
$z,u_1,\ldots,u_k\in V(G)$,
\[
\Pr\bigl(Z=z,\,
U=(u_1,\ldots,u_k)\bigr)
=
\frac{1}{2e(G)}\frac{1}{d(z)^{k-1}}
\prod_{j=1}^k\mathbf 1_{\{zu_j\in E(G)\}}.
\]
Hence, for every $\bm u=(u_1,\ldots,u_k)\in V(G)^k$,
\begin{align*}
\Pr(U=\bm u)
&=
\sum_{z\in V(G)}\Pr(Z=z,U=\bm u)\\
&=
\frac{1}{2e(G)}
\sum_{z\in\bigcap_{j=1}^kN_G(u_j)}\frac{1}{d(z)^{k-1}}\\
&=
\frac{Z_k(\bm u)}{2e(G)}.
\end{align*}
Therefore,
\begin{align*}
\E\bigl[\norm{q_U}_\infty\bigr]
&=
\sum_{\bm u:Z_k(\bm u)>0}
\Pr(U=\bm u)\norm{q_{\bm u}}_\infty\\
&=
\frac1{2e(G)}
\sum_{\bm u:Z_k(\bm u)>0}\ 
\max_{z\in\bigcap_{j=1}^kN_G(u_j)}\ \frac{1}{d(z)^{k-1}}\\
&\le
\frac{N^k}{2e(G)\delta(G)^{k-1}}
\le
\frac{N^{k-1}}{\delta(G)^k},
\end{align*}
by the fact that there are at most $N^k$ ordered tuples $\bm u$, and each nonzero summand is at most $\delta(G)^{1-k}$ and the last inequality follows from
$2e(G)\ge N\delta(G)$.
\end{proof}
\begin{remark}
    For $k=r$, the bound in Lemma~\ref{lem:max-atom} becomes
$N^{r-1}/\delta(G)^r$, which is of constant order when
$\delta(G)$ is of order $N^{1-1/r}$. This is the source of the
exponent in Theorem~\ref{thm:main}.
\end{remark}

\subsection{The random star-flip homomorphism}\label{sec:construction}
Let $H$ be an $r$-star-flip graph and fix a presentation
$R=H_0\subset H_1\subset\cdots\subset H_m=H$ with live-tree families
$\cT_0,\ldots,\cT_m$. For each $i\in[m]$, recall that $x_i$ is the new
vertex, $T_i\in\cT_{i-1}$ is the chosen live tree,
$c_i\in V(T_i)$ is the duplicated vertex, and
$N_{T_i}(c_i)=\{a_{i,1},\ldots,a_{i,k_i}\}$.

We define a random map $\Phi:V(H)\to V(G)$. First sample $\Phi|_R$
according to $\mu_R$. Then expose $x_1,\ldots,x_m$ in the order of the
presentation. When $x_i$ is exposed, sample
\begin{equation}\label{eq:flip-sampling}
\Phi(x_i)\sim q_{U_i},\qquad
U_i=(\Phi(a_{i,1}),\ldots,\Phi(a_{i,k_i})),
\end{equation}
conditional on all previously exposed  images.

\begin{lemma}\label{lem:well-defined}
The random map $\Phi$ defined above is a well-defined 
homomorphism from $H$ to $G$.
\end{lemma}

\begin{proof}
Suppose that $x_i$ is about to be exposed. By induction, the
already-exposed map $\Phi|_{H_{i-1}}$ is a homomorphism. Since
$c_i a_{i,j}\in E(T_i)\subseteq E(H_{i-1})$ for every $j\in[k_i]$,
the vertex $\Phi(c_i)$ is adjacent to every entry of $U_i$. Thus the
entries of $U_i$ have a common neighbour, and hence
$Z_{k_i}(U_i)>0$. Therefore, the distribution in
\eqref{eq:flip-sampling} is well-defined.

Moreover, every vertex in the support of $q_{U_i}$ is adjacent to
every entry of $U_i$. Hence every edge incident with $x_i$ is mapped
to an edge of $G$. Induction on $i$ now shows that $\Phi$ is a
homomorphism from $H$ to $G$.
\end{proof}

The key invariance property is that every live tree retains the
branching-random-walk distribution.

\begin{lemma}\label{lem:live-tree}
For every $0\le i\le m$ and every $T\in\cT_i$, the restriction
$\Phi|_T$ is distributed according to $\mu_T$.
\end{lemma}

\begin{proof}
We proceed by induction on $i$. For $i=0$, every tree in $\cT_0$ is a
subtree of the seed tree $R$, so the assertion follows from
Lemma~\ref{lem:subtree}.
Suppose that the assertion holds after step $i-1$. Every old live tree
$T\in\cT_{i-1}$ contains only previously exposed vertices, so
$\Phi|_T$ remains distributed according to $\mu_T$.

It remains to consider the new live trees. Recall that
\[
\mathcal S_i
=
\{T\in\cT_{i-1}:c_i\in V(T)\text{ and }N_T(c_i)=N_{T_i}(c_i)\}.
\]
Fix $T\in\mathcal S_i$, and let $T^{x_i}$ be obtained from $T$ by replacing
$c_i$ with $x_i$. Let $Y_T$ denote the vector of images of the
vertices in $V(T)\setminus\{c_i\}$.
By the induction hypothesis, $\Phi|_T$ is distributed according to
$\mu_T$. Hence, by Lemma~\ref{lem:conditional-law}, conditional on
$Y_T$, the vertex $\Phi(c_i)$ is distributed according to $q_{U_i}$,
where
$
U_i=(\Phi(a_{i,1}),\ldots,\Phi(a_{i,k_i})).
$

On the other hand, by \eqref{eq:flip-sampling}, conditional on all
previously exposed images, $\Phi(x_i)$ is distributed according to
$q_{U_i}$. Since $U_i$ is determined by $Y_T$ and the distribution in
\eqref{eq:flip-sampling} depends on the previously exposed images only
through $U_i$, conditional on $Y_T$, the random variable $\Phi(x_i)$
is also distributed according to $q_{U_i}$.
Thus, conditional on $Y_T$, the random variables $\Phi(c_i)$ and
$\Phi(x_i)$ have the same distribution. Since the images of all other
vertices are unchanged, the random map on $T^{x_i}$ is distributed
according to $\mu_{T^{x_i}}$ under the natural isomorphism replacing
$c_i$ with $x_i$.

Finally, every newly declared live tree is a subtree of some
$T^{x_i}$ with $T\in\mathcal S_i$. The result therefore follows from
Lemma~\ref{lem:subtree}.
\end{proof}

\section{Proof of the main theorem}\label{sec:collisions}
In this section, we prove Theorem \ref{thm:main}. Fix an $r$-star-flip presentation
$R=H_0\subset H_1\subset\cdots\subset H_m=H
$
with live-tree families $\cT_0,\ldots,\cT_m$. For each $i\in[m]$, let
$T_i\in\cT_{i-1}$ be the live tree chosen at step $i$, let
$c_i\in V(T_i)$ be the duplicated vertex, and let $x_i$ be the new
vertex introduced at that step. Write
$
k_i=d_{T_i}(c_i),
$
so that the $i$th operation is a $k_i$-star flip and $1\le k_i\le r$.
Set
$
\ell=|V(R)|$ and $h=|V(H)|.
$
Since each of the $m$ flips introduces exactly one new vertex, we have
$h=\ell+m$.

\begin{proposition}\label{prop:exact-criterion}
With the above notation, let $G$ be an $N$-vertex graph with minimum
degree $\delta>0$. If
\begin{equation}\label{eq:criterion-exact}
\frac{\binom{\ell}{2}}{\delta}
+
\sum_{i=1}^m(\ell+i-1)\frac{N^{k_i-1}}{\delta^{k_i}}
<1,
\end{equation}
then $G$ contains a copy of $H$.
\end{proposition}

\begin{proof}
We show that the random homomorphism $\Phi:H\to G$ from
\Cref{sec:construction} is injective with positive probability. By Lemma~\ref{lem:seed-collision}, for each pair of distinct vertices
$u,v\in V(R)$,
$
\Pr\bigl(\Phi(u)=\Phi(v)\bigr)\le 1/\delta.
$
Since there are $\binom{\ell}{2}$ such pairs, the union bound implies
that the probability that two distinct vertices of $R$ receive the
same image is at most $\binom{\ell}{2}/\delta$.

Fix $i\in[m]$ and $y\in V(H_{i-1})$. Conditional on all images exposed
before $x_i$, we have
\[
\Pr\bigl(\Phi(x_i)=\Phi(y)\mid\text{previously exposed images}\bigr)
=
q_{U_i}\bigl(\Phi(y)\bigr)
\le
\norm{q_{U_i}}_\infty.
\]
The star in $T_i$ with centre $c_i$ and leaves
$a_{i,1},\ldots,a_{i,k_i}$ is a subtree of the live tree $T_i$.
Hence, by Lemmas~\ref{lem:live-tree} and~\ref{lem:subtree}, its
image is distributed according to $\mu_{S_{k_i}}$. Taking expectations in the preceding conditional inequality 
and applying Lemma~\ref{lem:max-atom}, we obtain
\[
\Pr\bigl(\Phi(x_i)=\Phi(y)\bigr)
\le
\frac{N^{k_i-1}}{\delta^{k_i}}.
\]
At step $i$, there are exactly $\ell+i-1$ possible choices for $y$.
Every unordered pair $\{u,v\} \subset V(H)$ with at least one endpoint outside $R$ has a unique later-added endpoint, say $x_i$. Thus every such pair is counted exactly once, namely when
its later-added endpoint $x_i$ is considered together with the other
endpoint $y\in V(H_{i-1})$. Therefore, by the union bound,
the probability that two distinct vertices of $H$ receive the same
image is at most
\[
\frac{\binom{\ell}{2}}{\delta}
+
\sum_{i=1}^m(\ell+i-1)\frac{N^{k_i-1}}{\delta^{k_i}}.
\]
By \eqref{eq:criterion-exact}, $\Phi$ is injective with positive
probability, so $G$ contains a copy of $H$.
\end{proof}

\begin{proof}[Proof of Theorem~\ref{thm:main}]
Fix an $r$-star-flip presentation of $H$ with the notation introduced
above. If $m=0$, then $H$ is a tree, and the standard greedy embedding
gives $\ex(n,H)=O_H(n)$.

Suppose that $m\ge1$, and set
\[
B:=\sum_{i=1}^m(\ell+i-1)
=\binom{h}{2}-\binom{\ell}{2}>0.
\]
Let
$
C:=(2B)^{1/r},
$
and let $F$ be an $n$-vertex graph with at least
$Cn^{2-1/r}$ edges. Set $d_0=e(F)/n$ and repeatedly delete a vertex whose current degree is less than $d_0$. The process leaves a nonempty subgraph $G$ with $\delta(G)\geq d_0$. Otherwise, summing the degrees at the moments of deletion would give $e(F)< nd_0=e(F)$, a contradiction. Thus, we obtain a subgraph $G$ with
\[
\delta(G)\ge d_0 \ge Cn^{1-1/r}.
\]
Write $N=\abs{V(G)}$ and $\delta=\delta(G)$.
Since $N\le n$, $\delta<N$, and $k_i\le r$ for every $i\in[m]$, we
have
$
{N^{k_i-1}}/{\delta^{k_i}}
\le
{N^{r-1}}/{\delta^r}.
$
Consequently,
\begin{align*}
\frac{\binom{\ell}{2}}{\delta}
+
\sum_{i=1}^m(\ell+i-1)\frac{N^{k_i-1}}{\delta^{k_i}}
\ \le\ 
\frac{\binom{\ell}{2}}{Cn^{1-1/r}}
+
B\frac{N^{r-1}}{\delta^r}\ \le\ 
\frac{\binom{\ell}{2}}{Cn^{1-1/r}}
+\frac12.
\end{align*}
Once $n$ is large enough, the last expression is less than $1$.
Proposition~\ref{prop:exact-criterion} therefore implies that $G$, and
hence $F$, contains a copy of $H$.
Thus
$
\ex(n,H)=O_H(n^{2-1/r}).
$
\end{proof}

\section{Applications}\label{sec:applications}

\subsection{Products of two trees}
For a tree $T$ rooted at $t_0$, define the depth of a vertex
$t\in V(T)$ by
$\operatorname{dep}_T(t)=\operatorname{dist}_T(t,t_0)$.

\begin{proposition}\label{prop:tree-product-flip}
For every pair of nontrivial trees $T$ and $S$, the graph
$T\Box  S$ is a $2$-star-flip graph.
\end{proposition}\begin{proof}
Choose roots $t_0\in V(T)$ and $s_0\in V(S)$. For every
$t\in V(T)\setminus\{t_0\}$, let $t^-$ denote the neighbour of $t$ on
the unique path from $t$ to $t_0$. Define $s^-$ analogously for every
$s\in V(S)\setminus\{s_0\}$. Take as the seed tree
$R=(T\times\{s_0\})\cup(\{t_0\}\times S)$.
The two coordinate trees intersect only at $(t_0,s_0)$, so $R$ is a
tree.

The vertices outside $R$ are precisely the pairs $(t,s)$ with
$t\ne t_0$ and $s\ne s_0$. We shall add each such vertex $(t,s)$ by
performing a $2$-star flip at $(t^-,s^-)$ with neighbour set
$\{(t^-,s),(t,s^-)\}$. It therefore remains to show that, when $(t,s)$
is processed, the path $(t^-,s)-(t^-,s^-)-(t,s^-)$ is live.

We add the remaining vertices in nondecreasing order of
$D(t,s):=\operatorname{dep}_T(t)+\operatorname{dep}_S(s)$, processing
vertices with the same value of $D(t,s)$ in an arbitrary order. To
verify that each prescribed flip is legal, fix a vertex $(t,s)$ when
it is processed. By induction on the depth sum, every vertex of depth
sum smaller than $D(t,s)$ has already been added by a legal
$2$-star flip.

Write the root paths from $t_0$ to $t$ and from $s_0$ to $s$ as
\[
t_0,t_1,\ldots,t_p=t
\qquad\text{and}\qquad
s_0,s_1,\ldots,s_q=s.
\]
Thus $p,q\geq 1$ and $D(t,s)=p+q$. The path
\[
(t_0,s_q),(t_0,s_{q-1}),\ldots,(t_0,s_0),
(t_1,s_0),\ldots,(t_p,s_0)
\]
is a subtree of the seed tree and is therefore live. We take this as the initial tracked path associated with $(t,s)$.
Starting from this initial tracked path, we update it successively as
the flips preceding the processing of $(t,s)$ are performed. Whenever a vertex $(t_i,s_j)$ with $1\le i\le p$ and $1\le j\le q$
is processed and the tracked path contains the consecutive subpath
\[
(t_{i-1},s_j)-(t_{i-1},s_{j-1})-(t_i,s_{j-1}),
\]
the prescribed flip adding $(t_i,s_j)$ is performed at
$(t_{i-1},s_{j-1})$ with neighbour set
$\{(t_{i-1},s_j),(t_i,s_{j-1})\}$. Since the tracked path is live and
$(t_{i-1},s_{j-1})$ has precisely these two neighbours in it, the
tracked path is one of the live trees updated at this step, regardless
of which live tree is chosen to perform the flip. Hence the live-tree
update rule declares the path obtained by replacing
$(t_{i-1},s_{j-1})$ with $(t_i,s_j)$ to be live. We take this new path
as the tracked path. Vertices $(t',s')$ outside
$\{(t_i,s_j):1\le i\le p,\ 1\le j\le q\}$ do not affect the argument,
since every previously live tree remains live after each flip.
The claim below shows that the required subpath is present whenever
the corresponding vertex is processed.
\begin{claim}\label{clm:tracked-live-path}
Let $2\le d\le p+q$. After all vertices $(t_i,s_j)$ with
$1\le i\le p$, $1\le j\le q$, and $i+j<d$ have been added by legal
$2$-star flips, the tracked path is live and contains, for every pair
$(i,j)$ with $1\le i\le p$, $1\le j\le q$, and $i+j=d$, the
consecutive subpath
\[
(t_{i-1},s_j)-(t_{i-1},s_{j-1})-(t_i,s_{j-1}).
\]
\end{claim}

Assuming the claim, apply it with $d=p+q$. The tracked live path then
contains
\[
(t_{p-1},s_q)-(t_{p-1},s_{q-1})-(t_p,s_{q-1})
\]
as a consecutive subpath. We may therefore perform a $2$-star flip at
$(t_{p-1},s_{q-1})$, adding $(t,s)$ adjacent precisely to
$(t_{p-1},s_q)$ and $(t_p,s_{q-1})$.

Proceeding through all vertices in nondecreasing order of their depth
sums adds every vertex of $T\Box S$. All edges of the seed tree are
present initially, and every other Cartesian-product edge is added
when its endpoint of larger depth sum is introduced. No other edge is
added. Hence the resulting graph is exactly $T\Box S$, proving that
$T\Box S$ is a $2$-star-flip graph.
\end{proof}

\begin{proof}[Proof of the claim]
We proceed by induction on $d$. For $d=2$, the assertion follows from
the initial subpath
\[
(t_0,s_1)-(t_0,s_0)-(t_1,s_0).
\]

Suppose that the assertion holds for some $d<p+q$. For every pair
$(i,j)$ with $i+j=d$, the outer induction hypothesis ensures that the
prescribed flip adding $(t_i,s_j)$ has already been performed legally.
At the moment when this flip is performed, the tracked path contains
\[
(t_{i-1},s_j)-(t_{i-1},s_{j-1})-(t_i,s_{j-1}),
\]
and the middle vertex $(t_{i-1},s_{j-1})$ has precisely the neighbour set used by the flip.
Consequently, the live-tree update rule also updates the tracked path,
replacing $(t_{i-1},s_{j-1})$ with $(t_i,s_j)$.
These updates do not interfere with one another. Their centres are
distinct and have depth sum $d-2$, whereas the endpoints of all
relevant subpaths have depth sum $d-1$. Thus the updates corresponding to pairs with $i+j=d$ do not interfere
with one another.

Now fix $(i,j)$ with $i+j=d+1$. If $i>1$, the flip adding
$(t_{i-1},s_j)$ makes it adjacent to $(t_{i-1},s_{j-1})$ in the
tracked path; if $i=1$, this adjacency belongs to the initial path.
Similarly, $(t_i,s_{j-1})$ is adjacent to
$(t_{i-1},s_{j-1})$, either by the flip adding it or by the initial
path when $j=1$. Moreover, $(t_{i-1},s_{j-1})$ is not replaced in this
layer, since its depth sum is $d-1$. Hence the tracked path contains
\[
(t_{i-1},s_j)-(t_{i-1},s_{j-1})-(t_i,s_{j-1}),
\]
proving the induction step.
\end{proof}
\begin{proof}[Proof of Theorem~\ref{cor:tree-products}]
By Proposition~\ref{prop:tree-product-flip} and \Cref{thm:main},
$\ex(n,T\Box  S)=O_{T,S}(n^{3/2})$. Since both trees contain an edge,
$T\Box  S$ contains a copy of $C_4$. Consequently, the classical estimate
$\ex(n,C_4)=\Theta(n^{3/2})$
\cite{KST54,Brown1966,ERS1966} gives the matching lower bound.
\end{proof}
\begin{remark}
A closer inspection of the argument based on
Proposition~\ref{prop:exact-criterion} gives a leading constant 
$
O(|V(T)|\,|V(S)|).
$ For trees with particular structures, this dependence can be further
improved using the more precise criterion, but we do not pursue such
refinements here.
\end{remark}
\subsection{One exceptional vertex of unbounded degree}
We next recover the classical all-but-one bounded-degree result.

\begin{proposition}\label{prop:exceptional-supergraph}
Fix an integer $r\ge2$. Let $L=(A,B;E)$ be a bipartite graph with no isolated vertices. Suppose
that there is a vertex $b_0\in B$ such that $d_L(b)\le r$ for every
$b\in B\setminus\{b_0\}$. Then $L$ is a subgraph of an
$r$-star-flip graph on the same vertex set.
\end{proposition}

\begin{proof}
Add all missing edges between $b_0$ and $A$, and denote the resulting
graph by $L^+$. Take as the seed tree the star with centre $b_0$ and
leaf set $A$. For each $b\in B\setminus\{b_0\}$, the substar induced by
$\{b_0\}\cup N_L(b)$ is a subtree of the seed and hence is
live. Duplicate its centre $b_0$. The new vertex has neighbourhood
precisely $N_L(b)$, and the flip has size $d_L(b)\le r$. Performing
these flips for all $b\ne b_0$ produces $L^+$. Since $L\subseteq L^+$,
the result follows.
\end{proof}

\begin{proof}[Proof of Corollary~\ref{cor:exceptional-degree}]
Deleting isolated vertices does not affect the Tur\'an number for all
sufficiently large $n$, so we may assume that $L$ has no isolated
vertices. By Proposition~\ref{prop:exceptional-supergraph}, the graph $L$ is a
subgraph of an $r$-star-flip graph $L^+$. Therefore,
$\ex(n,L)\le\ex(n,L^+)=O_L(n^{2-1/r})$ by \Cref{thm:main}.
\end{proof}

\section*{Acknowledgements}
Lanchao Wang was supported by the NSFC under grant number 12471327, the
National Key R\&D Program of China under grant number 2024YFA1013900,
the China Scholarship Council, and the Institute for Basic Science
(IBS-R029-C4). Caihong Yang was supported by the Scientific Research
Funds at China University of Geosciences (Wuhan) (Project No.2026039).

\bibliographystyle{abbrv}
\bibliography{main}
\end{document}